\documentclass[11pt]{article}
\usepackage[utf8]{inputenc}
\usepackage{blindtext}
\usepackage{soul}
\usepackage{latexsym} 
\usepackage{amsmath,amssymb,amsfonts,amsthm} 
\usepackage{graphics,psfrag} 
\usepackage{fancyhdr,lastpage} 
\usepackage{color,verbatim} 
\usepackage[hidelinks]{hyperref} 
\usepackage[table]{xcolor} 
\usepackage{metalogo}
\usepackage{etoolbox}
\usepackage{cite}
\usepackage{mathrsfs}
\usepackage[affil-it]{authblk}
\usepackage{mathtools}

\usepackage{xfrac}
\usepackage{bbm}
\patchcmd{\thebibliography}{\section*{\refname}}{}{}{}

\newtheorem{theorem}{Theorem}[section]
\newtheorem{corollary}{Corollary}[theorem]
\newtheorem{lemma}[theorem]{Lemma}
\newtheorem{assumption}[theorem]{Assumption}

\newtheorem{proposition}[theorem]{Proposition}
\theoremstyle{definition}
\newtheorem{definition}[theorem]{Definition}
\newtheorem{example}[theorem]{Example}
\newtheorem{remark}[theorem]{Remark}

\hypersetup{
    colorlinks=false,
    pdfborder={0 0 0},
}

\usepackage{calc}

\usepackage[a4paper,top=1in,bottom=0.8in,right=0.8in,left=0.8in]{geometry}
\let\oldbibliography\thebibliography
\renewcommand{\thebibliography}[1]{%
  \oldbibliography{#1}%
  \setlength{\itemsep}{-1.5mm}%
}

\def\R{\mathbb{R}}

\def\N{\mathbb{N}}
\def\P{\mathbb{P}}

\newcommand{\be}{\begin{equation}}
\newcommand{\ee}{\end{equation}}
\newcommand{\bea}{\begin{eqnarray}}
\newcommand{\eea}{\end{eqnarray}}
\newcommand{\beann}{\begin{eqnarray*}}
\newcommand{\eeann}{\end{eqnarray*}}
\newcommand{\benn}{\begin{equation*}}
\newcommand{\eenn}{\end{equation*}}

\renewcommand{\P}{\mathbb{P}}

\fancypagestyle{plain}{%
  \fancyhf{}%
  \fancyfoot[C]{\footnotesize Page \thepage\ of \pageref{LastPage}}%
}

\title{Spectral theory of dense hypergraph limits}
\author[1,2]{Ágnes Backhausz\thanks{\texttt{backhausz.agnes@renyi.hu}}}
\author[3,4]{Christian Kuehn\thanks{\texttt{ckuehn@ma.tum.de}}}
\author[1,2]{Sjoerd Bastiaan van der Niet\thanks{\texttt{van.der.niet.sjoerd@renyi.hu}}}
\author[5,6]{Giulio Zucal \thanks{\texttt{zucal@mpi-cbg.de}}}

\affil[1]{ELTE E\"otv\"os Lor\'and University, Faculty of Science, Institute of Mathematics, P\'azm\'any P\'eter s\'et\'any 1/c, Budapest, Hungary}
\affil[2]{HUN-REN Alfr\'ed R\'enyi Institute of Mathematics, Re\'altanoda utca 13-15., Budapest, Hungary}
\affil[3]{Munich Data Science Institute, Walther-Von-Dyck Str.~10, 85748 Garching b.~M\"unchen, Germany}
\affil[4]{{Technical University of Munich, School of Computation Information and Technology, Department of Mathematics, Boltzmannstr.~3, 85748 Garching b.~M\"unchen, Germany}}
\affil[5]{Max Planck Institute of Molecular Cell Biology and Genetics, Dresden, Germany}
\affil[6]{Center for Systems Biology Dresden, Germany}

\date{November 6, 2025}

\begin{document}
\maketitle

\begin{abstract}

In this work, we develop a spectral theory for hypergraph limits. We prove the convergence of the spectra of adjacency and Laplacian matrices for hypergraph sequences converging in the $1$-cut metric. On the other hand, we give examples of matrix operators associated with hypergraphs whose spectra are not continuous with respect to the $1$-cut metric. Furthermore, we show that these operators are continuous with respect to other cut norms.

    \vspace{0.2cm}
\noindent {\bf Keywords:}  Graph limits, hypergraphs, spectral graph theory, graph homomorphism, Laplacian. 

    \vspace{0.2cm}
\noindent {\bf  Mathematics Subject Classification Number:}  05C65, 37A30. 
\end{abstract}

\section{Introduction}

Large networks are ubiquitous in applications, for example in neurobiology, economics, urban systems, epidemiology and electrical power grids. Networks are used to represent interactions between agents, pairwise (between two agents) or higher-order (between multiple agents at the same time). The typical mathematical objects to represent pairwise interactions are graphs. However, in many applications, the networks considered are extremely large. For this reason, recent years have seen a rapid development of graph limit theory \cite{LovaszGraphLimits}, where the limit objects are typically idealised analytical objects encoding only the relevant information about a large network. 

Graph limit theory for dense simple graph sequences (sequences where the number of edges grows proportionally to the square of the number of vertices) is the most well-developed area of the field \cite{BORGS20081801,Lovsz2007SzemerdisLF,LOVASZ2006933,borgs2011convergentAnnals}.
The theory for bounded-degree sequences of graphs, where each vertex has a uniformly bounded number of neighbours, has also been extensively studied \cite{BenjaminiLimit,local-global1,Hatami2014LimitsOL}.
By contrast, the limit theory for graphs of intermediate density remains much less understood, though it has recently attracted substantial attention \cite{BORGS20081801,borgs2011convergentAnnals,MarkovSpaces,KUNSZENTIKOVACS20191,frenkel2018convergence,backhausz2018action,veitch2015classrandomgraphsarising,janson2016graphons,caron2017sparse,borgs2018sparse,borgs2020identifiability,borgs2019sampling,10.1214/18-AOS1778,JANSON2022103549}. In addition, limits of weighted and edge-decorated graphs have been investigated in various contexts \cite{lovász2010limits,falgasravry2016multicolour,rath2011multigraph} and these limit objects have received increasing interest in recent years \cite{KUNSZENTIKOVACS2022109284,athreya2023pathconvergencemarkovchains,abraham2023probabilitygraphons,zucal2024probabilitygraphonsrightconvergence,zucal2024probabilitygraphonspvariablesequivalent}; see also the recent survey article \cite{ganguly2025multiplexonslimitsmultiplexnetworks}. 

For the classical theory for dense graph sequences \cite{BORGS20081801, LOVASZ2006933,borgs2011convergentAnnals} the limit objects are called graphons (from graph and function). 
A graphon is a (measurable) function $W$ from $[0,1]^2$ to $[0,1]$ and the convergence defined on these objects is given by the cut metric $\delta_{\square}$, see equation \eqref{eq:cutMetr}, that in turn is defined starting from the cut norm $\|\cdot\|_{\square}$, see equation \eqref{eq:def_NcutR}. Convergence in cut metric is also equivalent to the combinatorial point of view of the convergence of homomorphism densities \eqref{eq:homKernelformula}.

Recently, understanding higher-order interactions and the different phenomena caused by them has attracted  a lot of attention in physics and network science \cite{HigerOrdIntBook,majhi2022dynamics,MulasKuehnJost,bick_2023_higher}. To represent these interactions the natural combinatorial objects are hypergraphs. However, the development of analogous limit theories for hypergraphs is still very limited. Exceptions are the early works \cite{HypergraphonsZhao,hypergrELEK20121731,HypergraphsSzegedy2} and the recent works \cite{zucal2023action} and Section 9.4 in \cite{zucal2024probabilitygraphonspvariablesequivalent}. In this work, we develop further hypergraph limit theory for dense sequences, the limit objects in this case are called hypergraphons. For simplicity we focus in this introduction on the case of $3$-uniform hypergraphons, that are the limit objects for $3$-uniform hypergraphs. A $3$-uniform hypergraphon, is a (measurable) function $W$ from $[0,1]^6$ to $[0,1]$. In particular, for hypergraphons, differently from graphons, it is known that one has multiple possible choices for cut norms, capturing different combinatorial properties in the limit \cite{HypergraphonsZhao}. We understand further the combinatorial properties captured by the convergence in different cut norms, in the case of $3$-uniform hypergraphons the $1$-cut norm $\|\cdot\|_{\square,1}$, see equation \eqref{eq:one cut norm 2}, and the $2$-cut norm $\|\cdot\|_{\square,2}$, see equation \eqref{eq:CutNorm2_3unifhyp}. We consider hypergraphon contractions to graphons, that are directly related to contractions of adjacency tensors of hypergraphs to matrices related to some suitable underlying graph (Definition~\ref{def:Codegree2Sect} and Equation~\ref{HypergraphMatrixDifferent}). In particular, we understand which contractions of hypergraphons are continuous in cut norm $\|\cdot\|_{\square}$ with respect to which cut norm $\|\cdot\|_{\square,i}$ for $i=1$ or $2$ (Lemma~\ref{lemm:ContinuCodegreeGraphon} and Lemma~\ref{lemm:ContinuContrStrHypergraphon}). We also identify homomorphism densities of  hypergraphs with homomorphism densities of the related graphs obtained by the respective contraction of the adjacency tensor of the hypergraphs (Proposition~\ref{prop:subdivision counts} and Proposition~\ref{lemm:homDensHypIntersPatt}).  

As a key consequence of our results we obtain the pointwise convergence of the spectrum of several classical operators considered in spectral hypergraph theory under convergence in $1$-cut norm $\|\cdot\|_{\square,1}$.  Spectral graph theory \cite{chung1997spectral,brouwer2011spectra} studies properties of graphs in relation to the eigenvalues of matrices related to them and plays a key role in stochastic processes and combinatorics. Particularly important matrices considered in spectral graph theory are adjacency and Laplacian matrices. Graphons can be thought as natural continuum limits of adjacency matrices and Laplacian matrices for graphons also attracted considerable attention very recently \cite{RandomWalkGraphonLambiotte,LaplacianGraphonVizuete,klus2025learninggraphonsdatarandom}. Convergence of the spectrum of graphons has been studied in \cite{borgs2011convergentAnnals} and very recently in \cite{grebík2025convergencespectradigraphlimits} in the nonsymmetric case, see also \cite{hladký2025digraphonsconnectivityspectralaspects}. Another important application of spectral graph theory are dynamical systems on graphs. In this context, one is often faced with the standard dynamical question of linearisation around a steady state to determine stability. On the linearised level, the graph Laplacian is frequently appearing as an operator and one would like to determine what influence the graph structure has on stability. This concept was already discovered in the 1970s~\cite{SegelLevin,OthmerScriven}. Later on, this strategy became very popular in the onset of a wave of activity in complex network dynamics~\cite{PecoraCarroll} in the 1990s. 

Spectral graph theory is part of the broader framework of studying the spectra of more general discrete structures, such as hypergraphs \cite{JostMulasBook,MHJ,battiston_2021_physics,JOST2019870,MULAS202226,PhysRevE.101.022308,BANERJEE202182}. This area has seen an incredibly rapid development in recent years. Very recently, the same approach to network dynamics was also considered for dynamics on hypergraphs~\cite{MulasKuehnJost,carletti2020dynamical,MULAS202226}. Therefore, it is natural to try to understand spectra also for hypergraph limits.  However, this is the first work where the spectrum of limits of hypergraphs is systematically studied. In particular, we establish the convergence of the spectrum of adjacency and Laplacian matrices of hypergraph sequences (and hypergraphon sequences) under convergence in $1$-cut metric $\delta_{\square,1}$ (derived by the $1$-cut norm), Corollary~\ref{cor:ConvergenceSpectrHyp}. On the way, we establish also the pointwise convergence of the spectrum of the random walk Laplacian for graphons in cut distance, Theorem~\ref{ThmSpecLapGraph}, which has also not appeared in the literature before to the best of our knowledge.   We also establish that the convergence of the spectrum of less classical matrices for hypergraphs require convergence in $2$-cut norm instead of $1$-cut norm convergence.

\section{Graphs, graphons and Laplacians}\label{section:background}
A graph can be encoded via several matrix representations. We start by introducing perhaps the most well-known one. Let $G=(V,E)$ be a simple graph on $N$ vertices with vertex set $V$ and edge set $E$. The \emph{adjacency matrix} of $G$ is the matrix $A:=A(G)$ labelled by the vertices of $G$, and
\begin{equation*}
    A_{u v}= \begin{cases}1, & \text {if }\{u, v\} \in E, \\ 0, & \text {otherwise. }\end{cases}
\end{equation*}
In the following sections we also speak of edge-weighted graphs and their adjacency matrices.
For a weighted graph $G=(V,E,w)$, where $w:E\to\mathbb R_{> 0}$ is the weight function of the edges, the adjacency matrix $A:=A(G)$ has its entries defined as
\begin{equation*}
    A_{u v}= \begin{cases}w(u, v), & \text {if }\{u, v\} \in E, \\ 0, & \text {otherwise. }\end{cases}
\end{equation*}
We may write $G=(V(G),E(G))$ for clarity, if this is required by the context.

Applications involving graphs often analyse the properties of a Laplacian matrix. We restrict ourselves to the \emph{normalised Laplacian}, which is the matrix 
\begin{equation*}
    L=\mathrm{Id}-D^{-1 } A, 
\end{equation*}
where $\mathrm{Id}$ is the $N \times N$ identity matrix and $D$ is the diagonal matrix with the degree vector $(\deg v)_v$ on the diagonal. For a weighted graph, the degree of $v$ equals to the sum of the weights of all edges adjacent to $v$. The normalised Laplacian is also referred to as the \emph{random walk Laplacian} due to its following interpretation. For $u\neq v$, we have
\begin{equation*}
    -L_{u v}=\frac{A_{u v}}{\operatorname{deg} u},
\end{equation*}
which is the probability that a random walker on the vertices of $G$ moves from $u$ to $v$, where we assume the walker chooses the next vertex uniformly at random from the neighbours of its current state. We refer to \cite{chung1997spectral,MHJ} for more details on the normalised Laplacian and its spectral theory.

In order to compare graphs --- especially in the context of sequences of large dense graphs --- we make heavy use of the notion of graph homomorphisms. For two graphs $F$ and $G$, a \emph{graph homomorphism} from $F$ to $G$ is a function $\phi: V(F) \rightarrow V(G)$ such that for all $u,v\in V(F)$, if $\{u,v\}\in E(F)$, then $\{\phi(u),\phi(v)\}\in E(G)$. We denote by $\operatorname{hom}(F,G)$ the number of homomorphisms from $F$ to $G$. One easily verifies the following identity relating $\hom(F,G)$ to the adjacency matrix of $G$
\begin{equation}\label{eq:homGraphformula}
     \hom(F,G)=\sum_{\varphi: V(F)\rightarrow V(G)}\prod_{\{u,v\}\in E(F)}A_{\varphi(v)\varphi(w)}.
\end{equation}
\begin{example}\label{ExCycl}
It is well known that the $k$-th moment of the spectrum (spectral measure) of a graph’s adjacency matrix is equal to the number of closed walks of length $k$ in the graph, see for example \cite[Example 5.11]{LovaszGraphLimits}. This can be expressed in terms of homomorphism numbers. Let $C_k$ be a cycle graph of length $k$ and $A$ the adjacency matrix of a graph $G$. From equation \eqref{eq:homGraphformula} it follows that
\begin{equation*}
    \hom(C_k,G)= \operatorname{Tr}(A^k)=\sum^N_{i=1}\lambda_i^k
\end{equation*} 
where $\lambda_i$ are the $N$ eigenvalues of $A$.
\end{example}

Before moving on to graph limit theory, we introduce the following normalisation of $\operatorname{hom}(F,G)$.
For two graphs $F$ and $G$ the \emph{homomorphism density} $t(F, G)$ from $F$ to $G$ is
$$
t(F, G) = \frac{\hom(F, G)}{|V(G)|^{|V(F)|}},
$$
where we recall that $\hom(F, G)$ is the number of graph homomorphisms from $F$ to $G$. 

Homomorphism densities play a central role in the limit theory of dense graph limits. A sequence of graphs $(G_n)_n$ is said to be \emph{convergent} if the sequence of the homomorphism densities $(\hom(F, G_n))_n$ converges for each graph $F$. The natural limit objects to be considered for this type of convergence are graphons (from graph functions). For a more detailed background on the theory of dense graph limits we refer to  \cite{LovaszGraphLimits}. We briefly recall here the notion of a graphon. First, a \emph{kernel} is a bounded symmetric measurable function
\begin{equation*}
    W:[0,1]\times [0,1]\rightarrow \R.
\end{equation*}
A \emph{graphon} is a kernel
\begin{equation*}
    W:[0,1]\times [0,1]\rightarrow [0,1],
\end{equation*}
taking values in $[0,1]$. 

A graph $G$ on $N$ vertices with adjacency matrix $A$, can naturally be represented as a graphon $W_G$, defined as
\begin{equation*}
    W_G(x,y)=A_{\lceil Nx\rceil \lceil Ny\rceil}.
\end{equation*}

The definition of homomorphism densities for graphs extends naturally to homomorphism densities in kernels as well. Given a finite graph $F$ and a kernel $W$, define the homomorphism density $t(F, W)$ as

\begin{equation}\label{eq:homKernelformula}
    t(F, W) = \int_{[0,1]^{|V(F)|}} \prod_{\{u,v\}\in E(F)} W(x_u, x_v) \, \prod_{v\in V(F)}\mathrm{d} x_v.
\end{equation}
Similarly to graphs, we say that a sequence of kernels $(W_n)_n$ is convergent if the sequence $(t(F,W_n))_n$ converges for every graph $F$. Kernels receive their name from the fact that any such square-integrable kernel $W$ induces an operator $\mathcal{A}_W: L^2([0,1])\to  L^2([0,1])$, defined by
\begin{equation}\label{eq:kernel operator}
\begin{aligned}
    (\mathcal{A}_W f)(x)&=\int^1_0W(x,y)f(y)\,\mathrm{d}y,
    \end{aligned}
\end{equation}
where $L^2([0,1])$  denotes the space of real-valued square-integrable functions. For our purposes, this space is enough as we focus on self-adjoint operators. However, in more generality one could consider the space of  complex-valued square-integrable functions.

This gives the following analogy with Example~\ref{ExCycl}, identifying a graphon as a continuous extension of an adjacency operator.
\begin{example}\label{ExHomDensCycl}
Let $C_k$ be a cycle graph of length $k$ and $W$ any graphon. From Equation \eqref{eq:homKernelformula} it follows that \begin{equation*}
    t(C_k,W)=\sum_{\lambda \in \text{spec}(\mathcal{A}_{W})}\lambda^k,
\end{equation*} 
where $\text{spec}(\mathcal{A}_{W})$ denotes the spectrum of $\mathcal{A}_W$.\end{example}

We denote by $\|\cdot\|_{\square}$ the \emph{cut norm} defined as:
\begin{equation}
	\label{eq:def_NcutR}
	\|W\|_{\square} = \sup_{f,g} \bigg| \int_{[0,1]^2} W(x,y)f(x)g(y)\, \mathrm{d}x\, \mathrm{d}y \bigg|,
\end{equation}
for any kernel $W$, where the supremum is taken over measurable functions $f,g:[0,1]\to[0,1]$. The \emph{cut metric} is derived from this norm and is defined for two kernels $U$ and $W$ as 
\begin{equation}\label{eq:cutMetr}
    \delta_{\square}(U,W)=\inf_{\varphi}\|U-W^{\varphi}\|_{\square}.
\end{equation}
Here the infimum is over all measure preserving maps $\varphi:[0,1]\to[0,1]$, with $W^\varphi$ defined as $W^\varphi(x,y)\coloneq W(\varphi(x),\varphi(y))$, which results in $\delta_\square$ being only a pseudometric. We say that the graphons $U$ and $W$ are \emph{equivalent} if $\delta_\square(U,W)=0$. The following well-known theorems (see for example \cite{LovaszGraphLimits}) give us the compactness of the space of graphons (up to equivalence) equipped with the cut metric, and the equivalence between the two modes of convergence for graphons, see \cite[Theorem 9.23]{LovaszGraphLimits} and \cite[Theorem 11.5]{LovaszGraphLimits} respectively.
\begin{theorem}\label{thm:graphon compact}
    Any sequence of graphons $(W_n)_n$ admits a convergent subsequence with respect to $\delta_{\square}$.
\end{theorem}
\begin{theorem}\label{thm:graphon conv equiv}
A sequence of graphons $(W_n)_n$ is convergent with respect to $\delta_{\square}$ if and only if the sequence of homomorphism densities $(t(F,W_n))_n$ converges for every graph $F$.
\end{theorem}
We consider in later sections the limit properties of Laplace operators, which need not be symmetric. For this reason, we consider measurable functions $W:[0,1]^2\to[0,1]$, which are not necessarily symmetric. We refer to such a function as a \textit{$[0,1]$-valued kernel function}. In this paper, in the directed case, we require only the implication from convergence in cut distance to convergence of homomorphism densities, and no further theory will be needed. The definitions in this section concerning convergent graph sequences and graphons extends easily to the directed case. For this, one considers ordered pairs $(u,v)\in E(F)$ instead of unordered pairs, as in formulas \eqref{eq:homGraphformula} and \eqref{eq:homKernelformula}, and the test graphs $F$ in these formulas must be directed graphs without loops or more than one directed edge between two vertices. For such a directed graph $F$, and a $[0,1]$-valued kernel function $W$, the homomorphism density $t(F,W)$ can be defined in a similarly to \eqref{eq:homKernelformula}. As a result, the Counting Lemma for Graphons \cite[Lemma 10.23]{lovász2010limits} extends easily in the following way.
\begin{proposition}\label{prop:counting lemma}
Let $U,W:[0,1]^2\to[0,1]$ be two (not necessarily symmetric) measurable functions. Let $F$ be a directed graph without loops or more than one directed edge between two vertices. Then
    \begin{equation*}
        \lvert t(F,U)-t(F,W)\rvert\leq \lvert E(F)\rvert\delta_\square(U,W).
    \end{equation*}
\end{proposition}
\begin{proof}
  The proof in the symmetric case relies only on the absence of parallel (multiple) edges, and not on symmetry, see \cite[Lemma 10.23]{LovaszGraphLimits} or \cite[Theorem 4.5.1]{Zhao_2023}. Hence, the statement follows immediately.
\end{proof}

\subsection{Laplacian for graphons and its convergence}\label{section:laplacian convergence}
We prove here some convergence properties of the random walk kernel and random walk Laplacian for graphons that have been introduced (with slightly different conventions) in \cite{RandomWalkGraphonLambiotte}.

Let $W: [0,1]^2 \to [0,1]$ be a graphon. Define the degree function $d_W: [0,1] \to [0,1]$ by
\begin{equation*}
    d_W(x) = \int_0^1 W(x,y) \, \mathrm{d}y.
\end{equation*}
The \emph{random walk kernel} of a graphon $W$ is the function $K_W: [0,1]^2 \to \mathbb{R}$
\begin{equation*}
K_W(x,y) =\begin{cases}
      \frac{W(x,y)}{d_W(x)} \quad \text{if } d_W(x) > 0\\
     0 \qquad  \quad \text{if } d_W(x) = 0.
\end{cases}
\end{equation*}

It will be in many cases convenient to consider the following condition for graphons.

\begin{assumption}\label{AssumptionDeg}
 For the graphon $W$ there exists an $\varepsilon>0$ such that $d_{W}>\varepsilon$ almost everywhere.
\end{assumption}

\begin{remark}\label{RkRWkerngraphon}
If a graphon $W$ satisfies Assumption~\ref{AssumptionDeg} for a certain $\varepsilon>0$, then the kernel $\varepsilon K_W$ is a $[0,1]$-valued kernel function.
\end{remark}

Graphons whose degree functions are bounded from below almost everywhere exhibit a nice property, which is demonstrated in the following statement.
\begin{lemma}\label{lem:continuity graphon laplacian}
Let $W$ and $U$ be two graphons both satisfying Assumption~\ref{AssumptionDeg} for the same $\varepsilon>0$. Then the following inequality holds:
\begin{equation*}
     \|K_W-K_U\|_{\square}\leq \frac{2}{\varepsilon}\|W-U\|_{\square}.
\end{equation*}
\end{lemma}
\begin{proof}
    Write
    \begin{equation*}
        K_W(x,y)-K_U(x,y) = \frac{W(x,y)-U(x,y)}{d_W(x)} + \bigg(\frac{1}{d_W(x)} - \frac{1}{d_U(x)}\bigg)U(x,y),
    \end{equation*}
    then as $f(x)\varepsilon / d_W(x)$ maps into $[0,1]$, we have for the first term
    \begin{equation*}
    \begin{aligned}
        \sup_{f,g}\bigg|\int_{[0,1]^2}\frac{W(x,y)-U(x,y)}{d_W(x)}f(x)g(y)\,\mathrm{d}x\,\mathrm{d}y\bigg|&\leq\frac{1}{\varepsilon}\sup_{h,g}\bigg|\int_{[0,1]^2}W(x,y)-U(x,y)h(x)g(y)\,\mathrm{d}x\,\mathrm{d}y\bigg|\\
        \leq \frac{1}{\varepsilon}\|W-U\|_\square.
    \end{aligned}
    \end{equation*}
    For the second term we have 
    \begin{equation*}
        \int_{[0,1]^2}\bigg(\frac{1}{d_W(x)} - \frac{1}{d_U(x)}\bigg)U(x,y)f(x)g(y)\,\mathrm{d}x\,\mathrm{d}y = \int_0^1\frac{d_U(x)-d_W(x)}{d_W(x)d_U(x)}f(x)\int_0^1U(x,y)g(y)\,\mathrm{d}y\,\mathrm{d}x,
    \end{equation*}
    and since $\int U(x,y)g(y)\,\mathrm{d}y\leq d_U(x)$, we have that now $\frac{\varepsilon}{d_W(x)d_U(x)}f(x)\int U(x,y)g(y)\,\mathrm{d}y$ is a function of $x$ mapping into $[0,1]$. Thus we find
    \begin{equation*}
    \begin{aligned}
        \sup_{f,g}\bigg|\int_{[0,1]^2}\bigg(\frac{1}{d_W(x)} - \frac{1}{d_U(x)}\bigg)U(x,y)f(x)g(y)\,\mathrm{d}x\,\mathrm{d}y\bigg| &\leq\frac{1}{\varepsilon}\sup_{h}\bigg|\int_{[0,1]^2}(d_U(x)-d_W(x))h(x)\,\mathrm{d}x\,\mathrm{d}y\bigg|\\
        &\leq \frac{1}{\varepsilon}\|W-U\|_\square.
    \end{aligned}
    \end{equation*}
    This finishes the proof.
\end{proof}
Let $W$ be a graphon such that the random walk kernel $K_W$ of $W$ is square-integrable. We call the operator $\mathcal{K}_W:L^2([0,1])\to  L^2([0,1])$ induced by the random walk kernel the \emph{random walk kernel operator} and it is defined by $\mathcal{K}_W:=\mathcal{A}_{K_W}$, as in \eqref{eq:kernel operator}. The \emph{random walk Laplacian operator} $\mathcal{L}_W  :    L^2([0,1])\to  L^2([0,1])$ of $W$ is defined by
\begin{equation*}
\begin{aligned}
    (\mathcal{L}_W )f(x)&=f(x)- (\mathcal{K}_W )f(x).
    \end{aligned}
\end{equation*}
\begin{remark}\label{remarkEpsRW}
Let $W$ such that Assumption~\ref{AssumptionDeg} is satisfied, then from Remark~\ref{RkRWkerngraphon} it follows that
\begin{equation*}
\mathcal{K}_W \equiv \frac{1}{\varepsilon}\mathcal{A}_{\varepsilon K_W}.
\end{equation*}
\begin{remark}\label{RemarkL2dw}
If there exists an $\varepsilon>0$ such that $d_{W}>\varepsilon$ almost everywhere, then we can define the space $L^2([0,1],d_W)$ as the space of functions $f:[0,1]\to\mathbb{R}$ such that
\begin{equation*}
    \int^1_{0}|f(x)|^2d_W(x)\, \mathrm{d}x<\infty.
\end{equation*}
That is, the space of square-integrable functions with respect to the measure which is absolutely continuous with respect to the Lebesgue measure with density $d_W$. This space coincides with $L^2([0,1])$, the space of the square-integrable functions with respect to the Lebesgue measure. Moreover, the norms induced by the scalar products \begin{equation*}
    \langle f,g\rangle_2=\int^1_{0}f(x)g(x)\, \mathrm{d}x
\end{equation*}
and 
\begin{equation*}
    \langle f,g\rangle_{2,d_W}=\int^1_{0}f(x)g(x)d_W(x)\, \mathrm{d}x
\end{equation*}
induce equivalent norms on $L^2([0,1])$ in this case. 
\end{remark}
\end{remark}
\begin{remark}\label{RemarkSelfAdj}
It is easy to observe that the random walk Laplacian operator for a graphon $W$ is self-adjoint considered as an operator from the space $L^2([0,1],d_W)$ defined in Remark~\ref{RemarkL2dw} to itself and therefore the spectrum of this operator is real valued. We use the same symbol $\mathcal L_W$ for the operator on $L^2([0,1])$ and $L^2([0,1],d_W)$ with a slight abuse of notation. This directly implies that for a graphon $W$ for which Assumption~\ref{AssumptionDeg} is satisfied, the spectrum of $ \mathcal{L}_W$  as an operator from $L^2([0,1])$ to $L^2([0,1])$ is real valued. The same holds for the random walk kernel $\mathcal{K}_W$ of $W$. See also Proposition 6.2 in \cite{RandomWalkGraphonLambiotte} for more details. 
\end{remark}
\begin{remark}
    \label{RmkCompactOp}
For a graphon $W$ satisfying Assumption~\ref{AssumptionDeg}, the random walk kernel $\mathcal{K}_W$ of $W$ is a compact operator. In particular, $\mathcal{K}_W$ has a discrete spectrum, i.e.\ a countable multiset $\operatorname{spec}(W)$ of nonzero (real)  
eigenvalues $\{\lambda_1, \lambda_2, \ldots\}$ such that $\lambda_n \to 0$.  
In particular, every nonzero eigenvalue has finite multiplicity.
\end{remark}
The result of Lemma~\ref{lem:continuity graphon laplacian} gives a direct consequence for the convergence of a sequence of random walk kernels, for which their associated graphons are convergent. Moreover, we also obtain the pointwise convergence of the spectrum of their associated random walk Laplacian operators.
\begin{theorem}\label{ThmSpecLapGraph}
Let $(W_n)_n$ be a sequence of graphons converging to a graphon $W$, and assume there exists $\varepsilon>0$ such that $d_{W_n}>\varepsilon$ almost everywhere. Then the spectra of the random walk Laplacian operators $(\mathcal{L}_{W_n})_n$ and the spectra of the random walk kernel operators $(\mathcal{K}_{W_n})_n$, converge pointwise to the spectrum of $\mathcal{L}_{W}$ and $\mathcal{K}_W$, respectively. 
\end{theorem}
\begin{proof}
By Lemma~\ref{lem:continuity graphon laplacian} we obtain that the sequence $(K_{W_n})_n$ converges to $K_W$ with respect to $\delta_\square$. Then for every $k\geq3$ and $n\to\infty$
  \begin{equation*}
      \sum_{\lambda\in\text{spec}(\mathcal{K}_{W_n})}\lambda^k\to\sum_{\lambda\in\text{spec}(\mathcal{K}_{W})}\lambda^k,
  \end{equation*}
  since $t(\vec{C}_k,K_{W_n})\to t(\vec{C}_k,K_W)$, as $n\to\infty$, by Proposition~\ref{prop:counting lemma}. Here $\vec C_k$ denotes the directed cycle of length $k$. From this it follows that the spectra of $(\mathcal{K}_{W_n})_n$ converge pointwise to the spectrum of $\mathcal{K}_{W}$ (see the proof of~\cite[Theorem 11.54]{LovaszGraphLimits} for a more detailed argument and recall that the spectrum of $\mathcal{K}_{W_n}$ is real, see Remark~\ref{RemarkSelfAdj}). The same argument applies to $(\mathcal{L}_{W_n})_n$.
\end{proof}

\section{Hypergraphs}
A \emph{hypergraph} is a pair $H=(V,E)$ where $V$ is the set of \emph{vertices}, and $E$ is the set of \emph{edges} such that $\emptyset \neq e \subset V$ for each $e\in E$.
In particular, we consider two special cases of hypergraphs throughout the paper. A hypergraph $H$ is called \emph{$r$-uniform} if $|e|=r$ for every $e\in E$. It is called a \emph{linear hypergraph} if two distinct edges intersect in at most one vertex, i.e.\ $|e\cap f|\leq 1$ for all $e,f\in E$ with $e\neq f$. We define the \emph{degree} of $v\in V$ as $\deg(v)=|\{e\in E:\ v\in e\}|$. The \emph{codegree} of $u,v\in V$ is defined as $\text{codeg}_H(u,v)=|\{e\in E:\{u,v\}\subset e\}|$, the number of edges that contain both $u$ and $v$.
For example, graphs coincide with the $2$-uniform hypergraphs and for a graph $H$ we have $\text{codeg}_H(u,v)=1$ if $\{u,v\}$ is an edge and $0$ otherwise.

Similarly to ordinary graphs, we can define a \emph{homomorphism} between two hypergraphs as a map $\varphi:V(F)\rightarrow V(H)$, such that for each $e\in E(F)$ we have $\varphi(e)\coloneq \{\varphi(v):v\in e\}\in E(H)$.
Again, $\hom(F,H)$ is the number of homomorphisms between $F$ and $H$, and $t(F,H)=\hom(F,H) / |V(H)|^{|V(F)|}$ is the \emph{homomorphism density}.

In order to study graphs we will consider tensors that are natural generalisations of matrices. 
Let $r,N\geq 2$. An $r$-th order $n$-dimensional \emph{tensor} $T$ consists of $N^r$ entries
\begin{equation*}
    T_{i_1,\ldots,i_r}\in \mathbb{R},
\end{equation*}
where $i_1,\dots,i_r\in[n]$.

\subsection{Uniform hypergraphs}
We need to restrict ourselves to uniform hypergraphs for now, since this section introduces notions related to convergent sequences of hypergraphs (further developed in Section~\ref{section:uniform limits}). Most of these notions are natural generalisations of their counterparts from Section~\ref{section:background} to the setting of uniform hypergraphs.

A natural generalisation of the adjacency matrix for graphs, is the \emph{adjacency tensor} $A:=A(H)$ of an $r$-uniform hypergraph $H=(V,E)$. This is the $r$-th order $N$-dimensional tensor with entries defined as
\begin{equation*}
A_{v_1,\ldots,v_r}=\begin{cases}
1, & \text{if }\{v_1,\ldots,v_r\}\in E,\\
0, & \text{otherwise.}
\end{cases}
\end{equation*}
We aim to study random walk Laplacians on hypergraphs, which are matrices depending only on pairwise relationships between vertices.
For this reason we can define an appropriate weighted graph derived from a hypergraph and study the properties of this object, which are preserved under this transformation.
\begin{definition}[Codegree-section of a uniform hypergraph]\label{def:Codegree2Sect}
    The \emph{codegree-section} of the $r$-uniform hypergraph $H=(V,E)$ with $|V|=N$, denoted by $G[H]$, is an ordinary weighted complete graph without loops on the vertex set $V$ with weight function $w:E\to \mathbb R_{\geq 0}$, where the weight of each edge $\{u,v\}\in E$ is equal to $w(\{u,v\}):=N^{2-r}\text{codeg}_H(u,v)$. 
\end{definition}

\begin{remark}\label{rmk:linearHyp}
    Observe that if $H$ is linear, then for every $u,v\in V$ the codegree $\text{codeg}_H(u,v)$ is only $0$ or $1$. Therefore, the codegree-section $G[H]$ of $H$ is a graph with constant weight $N^{2-r}$ (basically a simple graph scaled by the constant $N^{2-r}$). However, we observe that linear hypergraphs have to be sparse.
\end{remark}

\begin{remark}
 The adjacency matrix $A$ of $G[H]$ has entries
 \begin{equation}\label{eq:WeightCod}
     A_{uv} = w(\{u,v\}) = \frac{\textup{codeg}_H(u,v)}{N^{r-2}},
 \end{equation}   
and it is sometimes referred to as the adjacency matrix of the hypergraph $H$.
\end{remark}

\begin{remark}
    Observe that for $u,v\in V$ in an $r$-uniform hypergraph we have
\begin{equation}\label{eq:codegree hypergraph}
    \text{codeg}_H(u,v)=\frac{1}{(r-2)!}\sum_{(w_2,\ldots, w_{r-1})\in V^{r-2}}A_{u,w_2,\ldots,w_{r-1},v},
\end{equation}
where we emphasise that $V^{r-1}$ denotes the Cartesian power of $V$. Similarly, we also have
\begin{equation*}
    \deg(v) =\frac{1}{(r-1)!}\sum_{(w_2,\ldots, w_{r})\in V^{r-1}}A_{v,w_2,\ldots,w_{r}}.
\end{equation*}
\end{remark}

\begin{example}\label{ex:uniform ER}
    Define the \textit{$r$-uniform Erd\H{o}s--R\'enyi graph} on $N$ vertices with parameter $p$ by the random $r$-uniform hypergraph on $N$ vertices where we include each of the possible $\binom{N}{r}$ possible edges with probability $p$, which we denote by $\mathbb G(N,p;r)$. Then the codegree of $\mathbb G(N,p;r)$ follows a binomial distribution with parameters $\binom{N}{r}$ and $p$. As a result, we have 
    \begin{equation*}
    \mathbb P(\text{codeg}_{\mathbb G(N,p;r)}(u,v) =  k) = \binom{\binom{N}{r}}{k} p^{k}(1-p)^{\binom{N}{r}-k}. 
    \end{equation*}
    Note that $\text{codeg}_{\mathbb G(N,p;r)}(u_1,v_1)$ and $\text{codeg}_{\mathbb G(N,p;r)}(u_2,v_2)$ for mutually distinct $u_1$, $v_1$, $u_2$ and $v_1$ are not independent apart from the the trivial case $r=2$. 
\end{example}
The homomorphism number for uniform hypergraphs admits a similar expression as the homomorphism number for graphs \eqref{eq:homGraphformula}, using the expression above.
Let $F=(V(F),E(F))$ and $H=(V(H),E(H))$ be two $r$-uniform hypergraphs and let $A$ be the adjacency tensor of $H$. We find the following identity
   \begin{equation*}
    \hom(F,H)=\sum_{\varphi: V(F)\rightarrow V(H)}\prod_{\{v_1,\ldots,v_r\}\in E(F)}A_{\varphi(v_1),\ldots,\varphi(v_r)},
   \end{equation*}
or alternatively, after relabelling the vertices of $F$ with $V(F)=[m]$ and the vertices of $H$ with $V(H)=[N]$, we have
\begin{equation*}
    \hom(F,H)=\sum_{i_1,\ldots, i_m\in [N]}\prod_{\{a_1,\ldots,a_r\}\in E(F)}A_{i_{a_1},\ldots,i_{a_r}}.
\end{equation*}
The above expression can be used to show a one-to-one correspondence between the homomorphism numbers $(\hom(F,G[H]))_F$ of $G[H]$ and $(\hom(F_r,H))_F$ of $H$, where each hypergraph $F^{(r)}$ is the linear $(r+2)$-uniform hypergraph obtained from the corresponding ordinary graph $F$ as explained in Definition~\ref{def:Rsubdiv}, and $H$ is also an $(r+2)$-uniform hypergraph. The construction is as follows.
\begin{definition}\label{def:Rsubdiv}
    An \emph{$r$-subdivision} of a graph $F=(V(F),E(F))$ is an $(r+2)$-uniform hypergraph $F_r=(V(F_r),E(F_r))$, where for every edge $e=\{u,v\}\in E(F)$ we construct a hyperedge $h(e)\in E(F_r)$ such that $e\subset h(e)$ and for distinct hyperedges $e,e'\in E(F_r)$, the sets $h(e)\setminus e$ and $h(e')\setminus e'$ are disjoint.
\end{definition}

\begin{remark}\label{rmk:SparsityHypergraphs}
It is easy to see that for a dense converging sequence of graphs the corresponding sequence of $r$-subdivisions for $r>0$ converges to $0$. This follows from the fact that $r$-subdivisions are linear hypergraphs and from Remark ~\ref{rmk:linearHyp} we know that sequences of linear hypergraphs are sparse. This shows that, for a sequence of hypergraphs $H_n$ converging to zero, a normalisation of $G[H_n]$ might have nontrivial limit. See also \cite{zucal2023action}.
\end{remark}

Because the edge weights of $G[H]$ are determined by the codegree function of the hypergraph $H$, which in turn can be expressed via \eqref{eq:codegree hypergraph} in terms of the adjacency tensor, we obtain the following proposition. 
\begin{proposition}\label{prop:subdivision counts}
Let $F$ be an ordinary graph and $H$ an $(r+2)$-uniform hypergraph. Then
\begin{equation}\label{eq:SubdivisionEquality}
   \hom(F_r,H)= (r! |V(H)|^r)^{|E(F)|}\hom(F,G[H]),
\end{equation}
where $F_r$ is the $r$-subdivision of the graph $F$ and $G[H]$ is the codegree-section of $H$.
\end{proposition}
\begin{proof}
Denote the vertices of $F$ and $H$ by $V(F)=[m]$ and $V(H)=[N]$ respectively, and the additional vertices introduced by constructing $F_r$ by $h_1(e),h_2(e),\ldots ,h_r(e)$, for each $e\in E$. We obtain the following expression for the RHS of \eqref{eq:SubdivisionEquality}
    \begin{equation*}
        \hom(F_r,H)=\sum_{i_1,\ldots, i_m\in [N]} \sum_{e\in E} \sum_{i_{h_1(e)},\ldots,i_{h_r(e)}\in[N]} \prod_{a\in E_r}A_{i_a},
    \end{equation*}
    where we write $A_{i_a}=A_{i_{a_1},\ldots,i_{a_{r+2}}}$ for $a=\{a_1,\ldots a_{r+2}\}\in E(F_r)$.
    Consider the hyperedge $h(e)\in E(F_r)$ corresponding to $e=\{u,v\}\in E(F)$, and suppose that $\{i_p:p\in h(e)\}\in E(H)$. Then the codegree of $i_u$ and $i_v$ is at least $1$. Or equivalently, in $G[H]$, the edge $\{i_u,i_v\}$ has nonzero weight, recall~\eqref{eq:WeightCod}. Under the condition $w(\{i_u,i_v\})>0$, we have $\{i_p:p\in h(e)\}\in E(H)$ if and only if the indices $i_{h_1(e)}$ through $i_{h_r(e)}$ correspond to a permutation of the remaining vertices of $\{i_p:p\in h(e)\}$.
    By construction none of the edges in $E(F_r)\setminus \{h(e)\}$ contain the internal vertices $h_1(e),\ldots,h_r(e)$ of $h(e)$. This gives the following expression 
    \begin{equation*}
        \sum_{i_{h_1(e)},\ldots,i_{h_r(e)}\in[N]} \prod_{a\in E(F_r)}A_{i_a}= \prod_{a\in E(F_r)\setminus h(e)} A_{i_{a}} \mathbbm{1}(w(\{i_u,i_v\})>0) \sum_{i_{h_1(e)},\ldots,i_{h_r(e)}=1}^n A_{i_{h(e)}},
    \end{equation*}
    where the last sum counts the number of shared edges between $i_u$ and $i_v$, including permutations of the internal vertices, and equals
    \begin{equation*}
        \sum_{i_{h_1(e)},\ldots,i_{h_r(e)}\in[N]} A_{i_{h(e)}} = r!\textup{codeg}_H(i_u,i_v).
    \end{equation*}
    The same argument holds for any edge of $F_r$, which results in
    \begin{equation*}
        \hom(F_k,H)=\sum_{i_1,\ldots, i_m\in [N]} \prod_{\{u,v\}\in E(F)} r! \textup{codeg}_H(i_u,i_v).
    \end{equation*}
    Here we recognise $A(G[H])_{i_u,i_v}=N^{-r}\textup{codeg}_H(i_u,i_v)$ as the weight of $i_u$ and $i_v$ in $G[H]$, with $A(G[H])$ the adjacency matrix of the 2-section $G[H]$. 
\end{proof}
\begin{remark}\label{rem:normalization codeg mat}
    From the preceding proposition we can deduce the following. Let $(H_n)_n$ be a sequence of $r$-uniform hypergraphs such that $(t(F,H_n))_n$ converges for every linear $r$-uniform hypergraph $F$. Then the sequence of graphs with adjacency matrices $A_n\coloneq A(G[H_n]) $ (recall~\eqref{eq:WeightCod}) is convergent in dense graph limits sense. Moreover, if $D_n$ is the diagonal matrix with the degrees of each vertex in $G[H_n]$ and satisfies $D_n>\varepsilon$ almost everywhere, then Lemma~\ref{lem:continuity graphon laplacian} implies the convergence of $(D_n^{-1}A_n)_n$. In the next section we introduce a metric $\delta_{\square,1}$ which captures the convergence of $(t(F,H_n))_n$ for all linear $r$-uniform $F$.
\end{remark}

The above remark highlights how the convergence of $(G[H_n])_n$ is implied by the convergence of $(H_n)_n$ with respect to $\delta_{\square, 1}$, and is related to convergence of the homomorphism densities $(t(F,H_n))_n$ for linear hypergraphs $F$. However, if one wishes to apply graph limit theory to matrix operators derived from hypergraphs, which arise from more complex relations between vertices, convergence with respect to $\delta_{\square, 1}$ is not enough in general. We will see this more in detail in the end of Section~\ref{section:uniform limits}.

\section{Uniform hypergraph limits}\label{section:uniform limits}
In this section we want to study limit objects of sequences of dense uniform hypergraphs. Since adjacency tensors generalise adjacency matrices for encoding uniform hypergraphs, it is natural to consider the following as the limit object for hypergraph sequences. We call a measurable symmetric function 
\begin{equation*}
    W:[0,1]^{r}\rightarrow [0,1]
\end{equation*}
an \emph{$r$-graphon}.

Similarly to the case of graphs, an $r$-uniform hypergraph $H$ on $N$ vertices with adjacency tensor $A$, can naturally be represented as an $r$-graphon $W_H$, defined as
\begin{equation*}
    W_H(x_1,\ldots, x_r)=A_{\lceil Nx_1\rceil,\ldots, \lceil Nx_r\rceil}.
\end{equation*}

\begin{remark}
An $r$-graphon is a trivial generalisation of graphons defined in Section~\ref{section:background}, which turns out to be the limit object for a sequence of $r$-uniform hypergraphs where the homomorphism densities for all linear $r$-uniform hypergraphs converge, as hinted in Remark~\ref{rem:normalization codeg mat}. In \cite{HypergraphonsZhao,hypergrELEK20121731,HypergraphsSzegedy2} the functions $W:[0,1]^{2^r-2}\rightarrow [0,1]$ are considered, which are shown to be the correct limit objects for sequences of $r$-uniform hypergraphs, where the homomorphism densities for any $r$-uniform hypergraph converge. In Example~\ref{ex:ER triangles the same} it is demonstrated how $r$-graphons are not able to encode certain structures. The additional coordinates represent all proper subsets of $[r]$. Their presence is linked to the need of suitable regularity partitions for hypergraphs \cite{GowersHypRegularity,RodlHyperReg1, RodlHyperReg2} and to the hierarchy of notions of quasi-randomness in the case of $r$-uniform hypergraphs  \cite{RandomneLimitTow}. However, for the moment we will focus only on $r$-graphons, the ``naive'' limit objects with $r$ coordinates. We will briefly discuss hypergraphons towards the end of this section.

\end{remark}
We measure the convergence of sequences of hypergraphs to $r$-graphons with the following norm. The \emph{$1$-cut norm} denoted by $\|\cdot\|_{\square,1}$ is defined as
\begin{equation}\label{eq:one cut norm 2}
\|U\|_{\square,1}=\sup_{f_1,\ldots,f_r}\bigg|\int_{[0,1]^r}U(x_1,\ldots,x_r)f_1(x_1)\cdots f_r(x_r) \mathrm{d}x_1\cdots \mathrm{d}x_r\bigg|,
\end{equation}
for an $r$-graphon $U$, where the supremum is taken over the measurable function $f_1,\ldots,f_r:[0,1]\to[0,1]$. The \emph{$1$-cut metric} for two $r$-graphons $U$ and $W$ is defined as
\begin{equation*}
\delta_{\square,1}(U,W)=\inf_{\varphi}\|U-W^{\varphi}\|_{\square,1},
\end{equation*}
where the infimum is again over all measure preserving maps $\varphi:[0,1]\to[0,1]$ and $W^{\varphi}$ is the $r$-graphon define as $W^{\varphi}(x_1,\ldots,x_r)\coloneq W(\varphi(x_1),\ldots, \varphi(x_r))$ for every $x_1,\ldots, x_r\in [0,1]$.

The space of equivalent $r$-graphons, where $U$ and $W$ are \emph{equivalent} if $\delta_{\square,1}(U,W)=0$, behaves well in the same sense as the space of equivalent graphons. We get the following analogue of Theorem~\ref{thm:graphon compact}, see for example \cite[Theorem 5.3]{LubetzkyZahoReplicaSymHyperg}.
\begin{theorem}\label{thm:r-graphon compact}
    Any sequence of $r$-graphons $(W_n)_n$ admits a convergent subsequence with respect to $\delta_{\square,1}$.
\end{theorem}

\begin{example}\label{ex:ER triangles the same}
    Consider an ordinary Erd\H{o}s--R\'enyi graph $\mathbb G(N,p;2)$, as in Example~\ref{ex:uniform ER}, and construct the following $3$-uniform hypergraph by including each possible edge $\{u,v,w\}$ whenever the triangle on the vertices $u$, $v$ and $w$ is present in $\mathbb G(N,p;2)$. We denote this random $3$-uniform hypergraph by $\mathbb T(N,p)$. It is a commonly used example that the sequences $(\mathbb T(n,p))_n$ and $(\mathbb G(n,p^3;3))_n$ converge almost surely to the same limit object, $W\equiv p^3$, with respect to the $1$-cut norm, while being structurally different. For a more detailed explanation see \cite[Example 23.11]{LovaszGraphLimits} and \cite[Section 1.2]{HypergraphonsZhao}.
\end{example}

Similarly to Theorem~\ref{thm:graphon conv equiv}, we find that the convergence of a sequence of $r$-graphons with respect to the $1$-cut metric is equivalent to the convergence of the homomorphism densities for all linear $r$-uniform hypergraphs \cite{HypergraphonsZhao}. For an $r$-graphon and an $r$-uniform hypergraph $F$, we define the homomorphism density $t(F,W)$ as
\begin{equation*}
    t(F,W) = \int_{[0,1]^{|V(F)|}} \prod_{\{v_1,\ldots,v_r\}\in E(F)} W(x_{v_1}, \ldots, x_{v_r}) \, \prod_{v\in V(F)}\mathrm{d} x_v.
\end{equation*}
The analogue of Theorem~\ref{thm:graphon conv equiv} follows from extending the Counting Lemma and Inverse Counting Lemma to $r$-graphons. In \cite[Theorem 5.5]{LubetzkyZahoReplicaSymHyperg} we find the Counting Lemma, i.e.\ Proposition~\ref{prop:counting lemma}, generalised to $r$-graphons and linear $r$-uniform test graphs $F$. Similarly, the Inverse Counting Lemma for $r$-graphons extends naturally by following the approach in \cite[Corollary 4.9.6]{Zhao_2023}.

We introduce the following graphon constructed from an $r$-graphon that coincides with the codegree-section of an $r$-uniform hypergraph as in Definition~\ref{def:Codegree2Sect}.

\begin{definition}[Codegree-section of an $r$-graphon]
    The \emph{codegree-section} of the $r$-graphon $W$ is the graphon $G[W]$ defined by
    \begin{equation*}
        G[W](x,y)=\frac{1}{(r-2)!}\int_{[0,1]^{r-2}}W(x,x_2,\ldots,x_{r-1},y)\,\mathrm{d}x_2\cdots\mathrm{d}x_{r-1}.
    \end{equation*}
\end{definition}

Now suppose that we have a sequence of $r$-graphons $(W_n)_n$. From Proposition~\ref{prop:subdivision counts} it can already be deduced that $(G[W_n])_n$ has a limiting graphon. Then the results from Section~\ref{section:laplacian convergence} apply and we obtain the spectral convergence of $(G[W_n])_n$ and the associated random walk Laplacians. However, at this point it could be unclear if the limit $r$-graphon $W$ (with respect to the $1$-cut distance) is such that $G[W]$ is equal to the limit of $(G[W_n])_n$. We show that convergence in the metric $1$-cut distance, that is equivalent to the convergence of all linear subgraph densities, implies that the limit objects are compatible. The result follows from the following lemma.
\begin{lemma}\label{lemm:ContinuCodegreeGraphon}
For two $r$-graphons $U$ and $W$ we have the following inequality
\begin{equation*}
    \|G[W]-G[U]\|_{\square}\leq \frac{1}{(r-2)!}  \|W-U\|_{\square,1}.
\end{equation*}
\end{lemma}
\begin{proof}
By observing that
\begin{equation*}
\begin{aligned}
    &\sup_{f,g}\bigg|\int_{[0,1]^r}(W(x,x_2,\ldots,x_{r-1},y) - U(x,x_2,\ldots,x_{r-1},y))f(x)g(y)\,\mathrm{d}x\mathrm{d}x_2\cdots\mathrm{d}x_{r-1}\mathrm{d}y\bigg|\\
    &\qquad\qquad\qquad\leq     \sup_{f_1,\ldots,f_r}\bigg|\int_{[0,1]^r}(W(x_1,\ldots,x_r) - U(x_1,\ldots,x_r))f_1(x_1)\cdots f_r(x_r)\,\mathrm{d}x_1\cdots\mathrm{d}x_{r}\bigg|,
\end{aligned}
\end{equation*}
where the suprema are taken over measurable functions from $[0,1]$ to $[0,1]$.
It follows that
\begin{equation*}
    \sup_{f,g}\bigg|\int_{[0,1]^2} \left(G[W](x,y)-G[U](x,y)\right)f(x)g(y)\,\mathrm{d}x\mathrm{d}y\bigg|\leq \frac{1}{(r-2)!}\|W-U\|_{\square,1}.
\end{equation*}
\end{proof}

As a result, we get that the convergence of a sequence of hypergraphs in the $1$-cut distance implies the convergence of the adjacency matrices in the cut distance, which in turn implies the convergence of the spectra of these matrices and the derived random walk Laplacian. 
\begin{corollary}\label{cor:ConvergenceSpectrHyp}
Let $(W_n)_n$ be a sequence of $r$-graphons converging to an $r$-graphon $W$ in $1$-cut metric $\delta_{\square,1}$, then the sequence of graphons $(G[W_n])_n$ converges to $G[W]$ in cut metric $\delta_{\square}$. In particular, the spectra of $(G[W_n])_n$ converge pointwise to the spectrum of $G[W]$.

Moreover, if there exists an $\varepsilon>0$ such that $d_{W_n}>\varepsilon$ for every $n$ then also the spectra of the random walk Laplacians $\mathcal{L}_{G[W_n]}$ of $(G[W_n])_n$ converge pointwise to the spectrum of the random walk Laplacian $\mathcal{L}_{G[W]}$ of $G[W]$.
\end{corollary}
According to Corollary~\ref{cor:ConvergenceSpectrHyp} it is sufficient to show convergence for a sequence of hypergraphs $(H_n)_n$ with respect to the $1$-cut metric in order for $(G[H_n])_n$ to converge with respect to the cut metric. 

\begin{example}\label{Examp2ErdosHypergraph}
    Consider the sequences $(\mathbb T(n,p))_n$ and $(\mathbb G(n,p^3;3))_n$ from Example~\ref{ex:ER triangles the same}. Almost surely, we have $d_{\mathbb T(n,p)}>\varepsilon$ and $d_{\mathbb G(n,p^3;3)}>\varepsilon$ almost everywhere for some $0<\varepsilon<p^3$. As a result, Corollary~\ref{cor:ConvergenceSpectrHyp} applies to $(\mathbb T(n,p))_{n\geq n_0}$ and $(\mathbb G(n,p^3;3))_{n\geq n_0}$. That is, although they are structurally different hypergraph sequences (as explained in Example~\ref{ex:ER triangles the same}), the random walk Laplacians of these sequences are indistinguishable in the limit, and so are their (pointwise) limiting spectra.
\end{example}

The previous example shows that some combinatorial information is lost in the limit when convergence in $1$-cut norm is considered. For this reason, there are matrices/operators for which the spectrum is not continuous with respect to the $1$-cut norm. We will show this with a simple example and we will discuss other convergence notions that capture the properties of this operators in the limit. 

Let us consider a $r$-uniform hypergraph $H=(V_H,E_H)$ with $|V_H|=n$ and its adjacency tensor $A$. The vertex-vertex intersection count matrix is the symmetric $n\times n$ matrix $B=B(H)$ with entries 
\begin{equation}\label{HypergraphMatrixDifferent}
    B_{uv}=\frac{1}{(r-1)!}\sum^n_{i_1,\ldots,i_{r-1}=1}A_{u,i_1,\ldots,i_{r-1}}A_{i_1,\ldots,i_{r-1},v}
\end{equation}
and the vertex-vertex intersection graph $B[H]$ is the graph with the matrix $B(H)$ as adjacency matrix. 

Observe that the entry $B_{uv}$ is the number of intersections of two hyperedges one containing $u$ and the other $v$ which have to intersect at least in all the $r-1$ elements of the two hyperedges except $u$ and $v$. In particular, $B_{uu}$ is the degree of the vertex $u$.

To understand the combinatorics behind these objects we show how certain homomorphism numbers of hypergraphs do relate to homomorphism numbers of their vertex-vertex intersection graphs, in a similar way to what we did in Proposition~\ref{prop:subdivision counts} for codegree-sections. In order to do this we will consider the following operation on graphs.

\begin{definition}
The $r$-\emph{intersection pattern hypergraph} of a graph $F$ is the $(r+1)$-uniform hypergraph $F^{(r)}=(V_r,E^{(r)})$ in which for every edge of the graph $e=\{v,u\}\in E$ we construct two edges of the hypergraph $h(e,v)=\{v, w_{e,1},\ldots ,w_{e,r}\}\in E^{(r)}$ and $h(e,u)=\{u, w_{e,1},\ldots ,w_{e,r}\}\in E^{(r)}$ where $w_{e,i}$ are pairwise different new elements for every $i\in [r]$ and $e\in E$ and $V_r=V\cup \left(\bigcup_{e\in E}\{ w_{e,1},\ldots ,w_{e,r}\}\right)$.
\end{definition}

The homomorphism numbers of intersection pattern hypergraphs into hypergraphs and homomorphism numbers of graphs are related in the following way. We omit the proof because the approach is identical to the proof of Proposition~\ref{prop:subdivision counts}. 

\begin{proposition}\label{lemm:homDensHypIntersPatt}
We have the following identity for the homomorphism numbers \begin{equation*}
    \hom(F^{(r)},H)=(r!)^{|E(F)|}\hom(F,B[H])
\end{equation*}
where we recall that $F^{(r)}$ is the intersection pattern hypergraph of $F$. 
\end{proposition}

In the rest of this section, in order to keep the explanation and the notation as simple as possible, we restrict ourselves to $3$-uniform hypergraphs. However, the discussion naturally extends to $r$-uniform hypergraphs for any $r\geq 3$.

We can consider the graphon representation of the vertex-vertex intersection graph for $r=3$. For a $3$-graphon $W$ the continuum version of~\eqref{HypergraphMatrixDifferent} is the vertex-vertex intersection graphon
\begin{equation}\label{eq:DefStrangeContrHype}
B(W)(x_1,x_4)=\frac{1}{2}\int_{[0,1]^2}W(x_1,x_2,x_3)W(x_2,x_3,x_4)\mathrm{d}x_2\mathrm{d}x_3.
\end{equation}
However, this graphon is not continuous with respect to the convergence in $1$-cut norm of $3$-graphons as shown by the following example. 

\begin{example}
As already remarked in Example~\ref{Examp2ErdosHypergraph}, the sequences $(\mathbb T(n,p))_n$ and $(\mathbb G(n,p^3;3))_n$ from Example~\ref{ex:ER triangles the same} both converge almost surely to the constant hypergraphon $W \equiv p^3$ in the $1$-cut norm. Therefore, $B(W)\equiv p^6$. However, for $(\mathbb T(n,p))_n$, which are the triangles of Erd\H{o}s--R\'enyi graphs, the sequence $B(W_{\mathbb T(n,p)})$ converges almost surely to the constant graphon $U\equiv p^5$. This can be seen by observing that both $A_{i,j,k}$ and $A_{j,k,\ell}$ are present in $\mathbb T(n,p)$ for distinct vertices $i$, $j$, $k$ and $\ell$, only if the five edges $\{i,j\}$, $\{j,k\}$, $\{i,k\}$, $\{k,\ell\}$, and $\{j,\ell\}$ are all present in the original Erd\H{o}s--R\'enyi graph. As a result $\P(A_{i,j,k}A_{j,k,\ell}=1)=p^5$. 
The results of \cite{Janson2004} can be applied to show that $(B(W_{\mathbb T(n,p)}))_n$ converges in cut norm to the constant graphon $U\equiv p^5$, and not the constant graphon $B(W)\equiv p^6$. Moreover, the spectrum of the operators $\mathcal{A}_{B(W)}$ and $\mathcal{A}_{B(U)}$ associated to $B(W)$ and $B(U)$ are clearly different. This shows that convergence in $1$-cut norm is not enough for the convergence in cut norm of the vertex-vertex intersection graphon from~\eqref{eq:DefStrangeContrHype} and also for the convergence of the spectrum of these graphons. 
\end{example}

For this reason, for vertex-vertex intersection graphons of $3$-graphons it is appropriate to consider a different metric. Actually, in order to introduce this metric we will have to consider an extension of a $3$-graphon that is a $3$-hypergraphon. A $3$-hypergraphon is a measurable function $W:[0,1]^6\rightarrow[0,1]$ such that \[
W(x_1,x_2,x_3,x_{12},x_{13},x_{23})=W(x_{\sigma(1)},x_{\sigma(2)},x_{\sigma(3)},x_{\sigma(1)\sigma(2)},x_{\sigma(1)\sigma(3)},x_{\sigma(2)\sigma(2)})
\]
for every permutation $\sigma$ acting on the set $\{1,2,3\}$. These have been shown to be the natural objects for the convergence of $3$-uniform hypergraphs. See~\cite{HypergraphonsZhao,HypergraphsSzegedy2} for more background on hypergraphons and the trivial generalisation of these objects to $r>3$.

Observe that a $3$-graphon can be trivially interpreted as a $3$-hypergraphon that is constant in the last three coordinates. Therefore, the vertex-vertex intersection graphon~\eqref{eq:DefStrangeContrHype} can be naturally generalised to $3$-hypergraphons. Let $W$ be a $3$-hypergraphon, the vertex-vertex intersection graphon is
\begin{equation*}
B(W)(x_1,x_4)=\frac{1}{2}\int_{[0,1]^7}W(x_1,x_2,x_3,x_{12},x_{13},x_{23})W(x_2,x_3,x_4,x_{23},x_{24},x_{34})\mathrm{d}x_2\mathrm{d}x_3\mathrm{d}x_{12}\mathrm{d}x_{13}\mathrm{d}x_{24}\mathrm{d}x_{23}.
\end{equation*}
We will consider the $2$-cut norm for $3$-uniform hypergraphons, that is 
\begin{equation}\label{eq:CutNorm2_3unifhyp}
\begin{aligned}
&\|W-U\|_{\square,2}\\&=\sup_{f,g,h}\Big|\int_{[0,1]^6}\left(W(x_1,x_2,x_3,x_{12},x_{13},x_{23})-U(x_1,x_2,x_3,x_{12},x_{13},x_{23})\right)\times\\          &\qquad \qquad f(x_1,x_2,x_{12})g(x_2,x_3,x_{23})h(x_1,x_3,x_{13})\mathrm{d}x_1\mathrm{d}x_2\mathrm{d}x_3\mathrm{d}x_{12}\mathrm{d}x_{13}\mathrm{d}x_{23}\Big|
\end{aligned}
\end{equation}
where the supremum is taken over all $f,g,h$ measurable functions from $[0,1]^3$ to $[0,1]$ satisfying the symmetry condition \begin{equation*}
f(x_1,x_2,x_{12})=f(x_{\sigma(1)},x_{\sigma(1)},x_{{\sigma(1)\sigma(2)}}),
\end{equation*} where $\sigma$ is any permutation of $\{1,2,3\}$. Again, see for example~\cite{HypergraphonsZhao} for more details and the $r>3$ case. Observe that these objects are $P$-variables as defined in \cite{zucal2024probabilitygraphonspvariablesequivalent}.

We will show that the vertex-vertex intersection graphon is continuous with respect to the $2$-cut norm $\|\cdot\|_{\square}$.
\begin{lemma}\label{lemm:ContinuContrStrHypergraphon}
Let $W$ and $U$ be two $3$-uniform hypergraphons. For the contractions $B(W)$ and $B(U)$, the following inequality holds
\begin{equation*}
    \|B(W)-B(U)\|_{\square}\leq \|W-U\|_{\square,2}.
\end{equation*}
\end{lemma}
\begin{proof}
Denote by
\begin{equation*}
    I(U,W)(x_1,x_4) = \frac{1}{2}\int_{[0,1]^7}U(x_1,x_2,x_3,x_{12},x_{13},x_{23})W(x_2,x_3,x_4,x_{23},x_{24},x_{34})\,\mathrm{d}(x_2,x_3,x_{12},x_{23},x_{24},x_{34}).
\end{equation*}
Observe that $B(U)\equiv I(U,U)$ and $I(U,U) - I(U,W)\equiv I(U,U-W)$. Thus
\begin{equation*}
    \|B(W)-B(U)\|_\square \leq \|I(U,U)-I(W,U)\|_\square + \|I(U,W)-I(W,W)\|_\square.
\end{equation*}
By rearranging the integral in the expression $\|I(U,U)-I(W,U)\|_\square=\|I(U-W,U)\|_\square$ we find
\begin{equation*}
\begin{aligned}
    &\int_{[0,1]^2} I(U-W,U)(x_1,x_4)f(x_1)g(x_4)\,\mathrm{d}x_1\mathrm{d}x_4\\
    &\qquad\qquad= \frac{1}{2}\int_{[0,1]^6}(U-W)(x_1,x_2,x_3,x_{12},x_{13},x_{23})f(x_1)F(x_2,x_3,x_{23})\,\mathrm{d}(x_1,x_2,x_3,x_{12},x_{13},x_{23}),
\end{aligned}
\end{equation*}
where $F(x_2,x_3,x_{23})=\int W(x_2,x_3,x_4,x_{23},x_{24},x_{34})g(x_4)\,\mathrm{d}(x_4,x_{24},x_{34})$. Thus $\|I(U-W,U)\|_\square\leq \frac{1}{2}\|U-W\|_{\square,2}$, and by symmetry we obtain the result.
\end{proof}
This directly implies the spectral convergence of the vertex-vertex intersection graphon. 
\begin{corollary}\label{Cor:2cutnormConvMatr}
Let $(W_n)_n$ be a sequence of $3$-hypergraphons converging to a  $3$-hypergraphon $W$ in $2$-cut norm $\|\cdot\|_{\square,2}$, then the sequence of graphons $(B(W)_n)_n$ converges to $B(W)$ in cut norm $ \|\cdot\|_{\square}$. Moreover, the spectrum of $(\mathcal{A}_{B(W)_n})_n$ converges to the spectrum of $\mathcal{A}_{B(W)}$.
\end{corollary}

At this point one could choose many more examples of contractions of adjacency tensors of $r$-uniform hypergraphs to matrices and study their convergence. Intuitively, one can think of all this matrices as 'building blocks' that have to be composed to construct the homomorphism densities of hypergraphs. In particular, in order to obtain the convergence of a specific matrix and its spectrum one has to choose a compatible metric.

\section{Extensions and applications}
In many applications, requiring the hypergraph to be uniform is too restrictive. We show how our results extend to nonuniform hypergraphs. As mentioned, it is often important to analyse a random walk process on a hypergraph, which is not restricted to be uniform. We recall three random walk processes on hypergraphs that have appeared in the literature. 
Each process is described by matrices $D$ and $A$, where the former is diagonal, such that
\begin{equation*}
P(v_i \to v_j)=\frac{A_{ij}}{D_{ii}}.
\end{equation*}

\begin{example}\label{ex:first RW} In \cite{PhysRevE.101.022308} the definitions are
\begin{equation*}
D_{ii}=\sum_{e\in E:v_i\in e}(|e|-1)\quad\mbox{and}\quad
A_{ij}=|\{e\in E:v_i,v_j\in e\}|.
\end{equation*}
\end{example}

\begin{example}\label{ex:second RW}
In \cite{BANERJEE202182} the definitions are
\begin{equation*}
D_{ii}=|\{e\in E:v_i\in e\}|\quad\mbox{and}\quad
A_{ij}=\sum_{e\in E:v_i,v_j\in e}\frac{1}{|e|-1}.
\end{equation*}
\end{example}

\begin{example}\label{ex:third RW}
In \cite{carletti2020dynamical} the definitions are
\begin{equation*}
D_{ii}=\sum_{j\ne i}\sum_{e\in E:v_i,v_j\in e}(|e|-1)\quad\mbox{and}\quad
A_{ij}=\sum_{e\in E:v_i,v_j\in e}(|e|-1).
\end{equation*}
\end{example}

In \cite{MULAS202226} it is shown that each of the above examples corresponds to a random walk process on a (not uniquely determined) weighted graph induced by the random walk process itself. However, when the hypergraph is uniform, the induced weighted graphs can be chosen identically. 

Suppose that $H=(V,E)$ is a (not necessarily uniform) hypergraph with largest edge cardinality $R\geq 2$. We say such a hypergraph has rank $R$. We may decompose $H$ into its \textit{levels}, the $r$-uniform hypergraphs $H^{(r)}$ for $2\leq r\leq R$, defined by the $r$-uniform hypergraph $H^{(r)}=(V,E^{(r)})$ on the vertex set $V$ and edge set $E^{(r)}=\{e\in E : |e|=r\}$. We call $H^{(r)}$ the \textit{$r$-th level} of $H$, and with a slight abuse of notation, we write $H=(H^{(2)},\ldots,H^{(R)})$.  We denote by $\textup{codeg}_H^{(r)}$ and $d_H^{(r)}$ the degree and codegree functions restricted to the $r$-th level of $H$. Moreover, if $(H_n)_n$ is a sequence of hypergraphs each of rank $R$ then the results on convergent hypergraph sequences apply to each of these sequences. 

The following application is aimed at Examples~\ref{ex:first RW}--\ref{ex:third RW}, and is explained further in Remark~\ref{rem:limit of RW examples}. Let $p=(p_r)_{2\leq r\leq R}$ be a probability vector with $p_R>0$. For a rank $R$ hypergraph $H$, the associated $p$-weighted adjacency matrix is defined by
\begin{equation}\label{eq:p-weighted matrix}
    A(H; p) = \frac{1}{N^{R-2}}\sum_{r=2}^Rp_r N^{r-2} G[H^{(r)}],
\end{equation}
where we slightly abuse notation by writing $G[H^{(r)}]$ for the adjacency matrix of the graph $G[H^{(r)}]$

If we interpret the above equation where the $G[H^{(r)}]$ represent $G[W_{H^{(r)}}]$, it follows that for large $N$, the term $p_R N^{R-2} G[H^{(R)}]$ dominates due to the normalisation by $N^{R-2}$. This construction gives the following proposition.

\begin{proposition}\label{prop:nonunif limit}
Let $(H_n)_n$ be a sequence of rank $R$ hypergraphs such that $|V(H_n)|\to\infty$ and $\delta_\square(G[H_n^{(R)}], W^{(R)})\to 0$ for some graphon $W^{(R)}$. Then $\delta_\square (A(H_n;p), p_RW^{(R)})\to 0$.
\end{proposition}
\begin{proof} We have
\begin{equation*}
    \delta_\square(A(H_n;p), p_RW^{(R)}) \leq \delta_\square(p_RG[H_n^{(R)}], p_RW^{(R)}) + \sum_{r=2}^{R-1}p_r |V(H_n)|^{r-R} \|G[H_n^{(r)}]\|_{\square},
\end{equation*}
where the terms $p_r\|G[H^{(r)}]\|_{\square}$ are bounded, thus the result follows.
\end{proof}

Moreover, if $(H_n^{(R)})_n$ converges with respect to the $1$-cut norm, then the limit of $(G[H_n^{(R)}])_n$ can be identified as $G[H^{(R)}]$. As a result we have $\delta_\square (A(H_n;p), p_RG[H^{(R)}])\to 0$.

\begin{remark}\label{rem:limit of RW examples}
Assume that the hypergraphs in Examples~\ref{ex:first RW}--\ref{ex:third RW} are of rank $R\geq 2$. Then we can represent the terms $D_{ii}$ and $A_{ij}$ in terms of linear combinations of $(d_H^{(r)}(v_i))_{2\leq r\leq R}$ and $(\textup{codeg}_H^{(r)}(v_i,v_j))_{2\leq r\leq R}$ respectively. For Example~\ref{ex:first RW} we have
\begin{equation*}
    D_{ii}=\sum_{r=2}^R (r-1)d_H^{(r)}(v_i)\quad\mbox{and}\quad A_{ij}=\sum_{r=2}^R \textup{codeg}_H^{(r)}(v_i,v_j),
\end{equation*}
for Example~\ref{ex:second RW} we have
\begin{equation*}
    D_{ii}=\sum_{r=2}^R d_H^{(r)}(v_i)\quad\mbox{and}\quad A_{ij}=\sum_{r=2}^R \frac{1}{r-1}\textup{codeg}_H^{(r)}(v_i,v_j),
\end{equation*}
and for Example~\ref{ex:third RW} 
\begin{equation*}
    D_{ii}=\sum_{r=2}^R (r-1)^2d_H^{(r)}(v_i)\quad\mbox{and}\quad A_{ij}=\sum_{r=2}^R (r-1)\textup{codeg}_H^{(r)}(v_i,v_j).
\end{equation*}
Since $\textup{codeg}_H^{(r)}=N^{r-2}G[H^{(r)}]$, one easily observes how to choose $p$ such that each of the preceding matrices $A$ is equal to $N^{R-2}A(H;p)$  up to a constant, with $A(H;p)$ from \eqref{eq:p-weighted matrix}. Hence, if $(H_n)_n$ is as in Proposition~\ref{prop:nonunif limit}, then for each of the above examples $A(H;p)$ converges to the same graphon with respect to the cut norm. Hence if we also have $d_{H_n}^{(R)}>\varepsilon$ for some $\varepsilon>0$ and $n\in \N$, then, by Theorem~\ref{ThmSpecLapGraph}, the associated random walk kernels all converge. In fact, they converge to the same limit 
\begin{equation*}\label{eq:nonuniform random walk limit}
    S_{W^{(R)}}(x_1,x_2)=\frac{1}{R-1}\frac{1}{d_{W^{(R)}}(x_1)}G[W^{(R)}](x_1,x_2).
\end{equation*}
\end{remark}

\begin{remark}       
    As pointed out in Remark~\ref{rmk:SparsityHypergraphs}, normalisations of the codegree-section graphs for sparse hypergraph sequences can still capture nontrivial limits. One could use this to define an analogue of \eqref{eq:p-weighted matrix} but extra care is needed now because the largest levels of the hypergraph sequences might interact in a nontrivial way with the other levels, differently from the dense case in which the largest levels dominate.
\end{remark}

\noindent{}\textbf{Acknowledgement.} \'A. B. and S. v.d. N. acknowledge support by the project “BeyondTheEdge: Higher-Order Networks and Dynamics” (European Union, REA Grant Agreement No. 101120085).

\section*{References}

\bibliographystyle{plain}
\bibliography{biblio}

\end{document}